\documentclass[11pt, a4paper, reqno]{amsart}
\usepackage{preambolo}

\usepackage[reqno]{amsmath}

\usepackage{biblatex}
\begin{document}

\setcounter{secnumdepth}{3}

\setcounter{tocdepth}{2}

\title[GENERALIZED LUTTINGER SURGERY AND OTHER CUT-AND-PASTE CONSTRUCTIONS]{\textbf{GENERALIZED LUTTINGER SURGERY AND OTHER CUT-AND-PASTE CONSTRUCTIONS IN GENERALIZED COMPLEX GEOMETRY}
}

\author[Lorenzo Sillari]{Lorenzo Sillari}

\address{LS: Scuola Internazionale Superiore di Studi Avanzati (SISSA), Via Bonomea 265, 34136 Trieste (TS), Italy.} \email{lsillari@sissa.it}

\maketitle

\begin{abstract} 
\noindent \textsc{Abstract}. Exploiting the affinity between stable generalized complex structures and symplectic structures, we explain how certain constructions coming from symplectic geometry can be performed in the generalized complex setting. We introduce generalized Luttinger surgery and generalized Gluck twist along $\J$-symplectic submanifolds. We also export branched coverings to the generalized complex setting. As an application, stable generalized complex structures are produced on a variety of high-dimensional manifolds. Remarkably, some of them have non-homotopy-equivalent path-connected components of their type change locus.
\end{abstract}

\blfootnote{  \hspace{-0.55cm} 
{\scriptsize 2020 \textit{Mathematics Subject Classification}. Primary: 53C15, 53D18; Secondary: 53D05, 57R17. }}

\section{Introduction}

A \emph{generalized complex structure} (\cite{Hit03, Gua03,Gua11}) $\J$ on a $2m$-manifold $M$ induces a symplectic structure on a subspace of $T_p M$ and a transverse complex structure. The dimension of the latter is an integer $t \in \{ 0, \dots , m\}$, called the \emph{type} of $\J$ at $p$. Symplectic structures have type $0$ at every point and they are the most generic situation, while complex structures have type $m$ at every point and they are the less generic one. \\
In general, the type might change on $M$, and the set of points where it is not locally constant is called the \emph{type change locus} of $\J$. The most generic case with non-empty type change locus is that of \emph{stable} generalized complex structures and was extensively studied in \cite{CG18}. Around points where $\J$ has type $0$ (i.e. outside of the type change locus), $M$ behaves like a symplectic manifold. This allows to consider ideas and techniques from symplectic geometry and import them in generalized complex geometry, leading to interesting consequences and fruitful applications (\cite{CG07, CG09, Tor12, TY14, Got16, GH16, BCG17, BCK18, CG18, CK18, BCD19, CK19,  CKW20, CKW22}).
\vspace{.2cm}

A fundamental notion in doing so is that of \emph{$\J$-symplectic submanifold} (\cite{GH16, Mun18}), the generalized complex analogous of a symplectic submanifold. The key observation is that as soon as we have a $\J$-symplectic submanifold, several techniques involving symplectic submanifolds extend with little or no effort to the generalized complex setting. In this paper, we address three of them, namely \emph{generalized Luttinger surgery}, a generalization of \emph{torus surgery} (\cite{CG07}), \emph{generalized Gluck twist}, a combination of generalized Luttinger surgery and \emph{Gluck twist} (\cite{Glu62}), and \emph{branched coverings}. However, this paper is not meant to cover every tool that can be obtained in this way, nor to exhaust all the possible examples that can be produced using them. Following our guidelines, the list of available constructions can be easily enlarged and many more examples will arise combining such techniques.
\vspace{.2cm}

In more detail, torus surgery has been fruitfully employed in generalized complex geometry in dimension $4$ (\cite{CG07, CG09, Tor12, TY14, GH16}) and $6$ (\cite{Mun18}) to obtain many examples of stable generalized complex structures. The surgery is performed along a $\J$-symplectic torus of codimension $2$ embedded in a stable generalized complex manifold.\\
Generalized Luttinger surgery extends torus surgery in two ways: first, it can be applied to manifolds of any even dimension $2m \ge 6$. This is the first occurrence of a similar construction in dimension $2m \ge 8$. Second, it modifies the type change locus by adding path-connected components that are homotopy-equivalent to a product $T^2 \times \Sigma$, where $\Sigma$ is a symplectic $(2m-4)$-manifold. This is new even in dimension $6$, where the only known possibilities were $T^4$ and $T^2 \times \S^2$ (\cite{Mun18}). More precisely, we prove the following

\begin{thmx}\label{main:surgery:intro}
Let $(M^{2m}, \J, H)$ be a stable generalized complex structure and let $S$ be its type change locus. Let $(\Sigma, \omega_\Sigma)$ be a symplectic $(2m-4)$-manifold. Assume that there exists an embedded $T^2 \times \Sigma \hookrightarrow M$ such that
\begin{itemize} 
    \item [$(i)$] there exists a neighborhood $N_T$ of $T^2 \times \Sigma$ that is diffeomorphic to $D^2 \times T^2 \times \Sigma$;  
    \item [$(ii)$] $T^2 \times \Sigma$ is $\J$-symplectic and $\omega_{|T^2 \times \Sigma} = d \theta^{12} + \omega_\Sigma$, where $\omega$ is the symplectic form induced by $\J$ on $N_T$, up to $B$-field equivalence.
\end{itemize}
Then performing generalized Luttinger surgery along $T^2 \times \Sigma$ yields a stable generalized complex structure $(\Tilde{M}, \Tilde{\J}, \Tilde{H})$ whose type changes along $S \sqcup \left( T^2 \times \Sigma \right)$.
\end{thmx}

The possibility of performing surgeries along $T^2 \times \Sigma$ is used to prove the following

\begin{thmx}
There exist stable generalized complex structures with non-homotopy-equivalent path-connected components of their type change locus.
\end{thmx}

We perform surgeries on stable generalized complex manifolds of the form $M \times T^2$ to obtain stable generalized complex manifolds of the form $\Tilde{M} \times \S^1$. A wise choice of the surgery parameters allows to produce manifolds $\Tilde{M}$ with a wide choice of fundamental groups. Under additional hypotheses, we pin down the diffeomorphism type of our $6$-dimensional examples.
\vspace{.3cm}

Generalized Luttinger surgery can be easily modified allowing more freedom for the gluing diffeomorphism. We describe how to do that in general. This paves the way to a possibly huge palette of constructions and enlarges the class of diffeomorphism types of manifolds admitting a stable generalized complex structure. Making an explicit choice of diffeomorphism, we introduce generalized Gluck twist. It extends from dimension $6$ to any even dimension $2m \ge 8$ the surgery described in \cite{Mun18}, Theorem 3.3.2. Results similar to those proved for generalized Luttinger surgery are true also for generalized Gluck twist.

\begin{thmx}
Let $(M^{2m}, \J, H)$ be a stable generalized complex structure and let $S$ be its type change locus. Let $(R, \omega_R)$ be a symplectic $(2m-6)$-manifold. Assume that there exists an embedded $T^2 \times R \times \S^2 \hookrightarrow M$ such that
\begin{itemize} 
    \item [$(i)$] there exists a neighborhood $N_T$ of $T^2 \times R \times \S^2$ that is diffeomorphic to $D^2 \times T^2 \times R \times \S^2$;  
    \item [$(ii)$] $T^2 \times R \times \S^2$ is $\J$-symplectic and $\omega_{|T^2 \times R \times \S^2} = d \theta^{12} + dh \, d\theta^S + \omega_R$, where $\omega$ is the symplectic form induced by $\J$ on $N_T$, up to $B$-field equivalence.
\end{itemize}
Then performing generalized Gluck twist along $T^2 \times R \times \S^2$ yields a stable generalized complex structure $(\Tilde{M}, \Tilde{\J}, \Tilde{H})$ whose type changes along $S \sqcup ( T^2 \times R \times \S^2)$.
\end{thmx}
\vspace{.3cm}

The third construction we study is that of branched coverings. The symplectic side of the problem was already settled in Lemma 1 of \cite{Gom98}. Rephrasing the result in the generalized complex language, we obtain that branched coverings of stable generalized complex $2m$-manifolds, with $2m \ge 4$, are still stable generalized complex.

\begin{thmx}
Let $(M^{2m}, \J,H)$ be a stable generalized complex structure whose type change locus has $k \in \N$ path-connected components. Let $B$ be a $\J$-symplectic submanifold of $M$, and let $\pi: \Tilde{M} \rightarrow M$ be a $d$-fold branched covering with branch locus $B$. Then there exists a stable generalized complex structure $(\Tilde{M}, \Tilde{\J}, \Tilde{H})$ whose type change locus has $d \cdot k$ path-connected components.
\end{thmx}

Combining generalized Luttinger surgery and coverings, we further exemplify the versatility of our ideas and their potential in producing stable generalized complex structures. Finally, using branched coverings we show a new way to put stable generalized complex structures on  $T^2 \times \Sigma_g $, where $\Sigma_g$ is the surface of genus $g \ge 1$, and on the elliptic surfaces $E(2n)$.

Even though we stated our results for stable generalized complex structures, they can be extended with no effort also to unstable structures (cf. Remark \ref{rem:unstable}).
\medskip

\textit{Acknowledgements.} This paper is the first part of the ongoing PhD project of the author. He is extremely grateful and deeply indebted to Rafael Torres, who provided the main results describing the techniques and most of the consequences we presented here. He also pointed out to the author many of the examples. This paper would not have been written without his expertise and patience. The author also thanks Roberto Rubio for comments that improved the exposition of the paper.

\section{Preliminaries}\label{section:preliminaries}

In this section we recall basic definitions and results in generalized complex geometry. Many of them can be found in \cite{Gua03}, while the others are referenced throughout the text. We also introduce our notation. Unless otherwise specified, using the word \emph{manifold} we mean a compact orientable smooth manifold with no boundary, and we shall often write $M^k$ to denote a $k$-manifold. If $\{ x_j \}$ are local (real or complex) coordinates, we will write $dx^{j_1 \dots j_k}$ referring to the differential form $dx^{j_1} \wedge \dots \wedge dx^{j_k}$.
\vspace{.3cm}

Let $M$ be a $2m$-manifold. Denote by $TM$, respectively $T^*M$, its tangent bundle, respectively cotangent bundle, and set $E := TM \oplus T^*M$.\\
Fixed any closed $3$-form $H$, the \emph{$H$-twisted Courant bracket} can be defined on sections of $E$ as follows:
\[
[X + \xi, Y + \eta]_H := [X,Y] + \mathcal{L}_X \eta - \mathcal{L}_Y \xi -\frac{1}{2} d \left( \eta(X) - \xi (Y) \right) + \iota_Y \iota_X H.
\]
Sections of $E$ also admit a natural symmetric pairing of signature $(n,n)$ defined by
\[
\langle X + \xi, Y + \eta \rangle := \frac{1}{2} \left( \eta(X) + \xi(Y) \right).
\]
\begin{definition}
A \emph{generalized complex structure twisted by $H$}, $(M,\J,H)$, is a maximal isotropic sub-bundle $L$ of $E^\C$ satisfying the following conditions:
\begin{itemize}
    \item \emph{Non-degeneracy:} $L \cap \bar L = \{ 0 \}$;
    \item \emph{Integrability:} $L$ is closed under the $H$-twisted Courant bracket.
\end{itemize}
Equivalently, $\J$ can be described in terms of a complex line bundle $K \subset \Omega^\bullet$, generated by a \emph{pure spinor}, i.e., an element $\rho \in \Omega^\bullet$ that has the local form
\[
\rho = e^{B+i\omega} \wedge \Omega,
\]
where $B$, and $\omega$ are real $2$-forms and $\Omega$ is a decomposable $k$-form. Moreover the pure spinor must satisfy the conditions:
\begin{itemize}
    \item \emph{Non-degeneracy:} the forms $\Omega$ and $\omega$ must provide a non-degenerate top-form
    \[
    \Omega \wedge \bar \Omega \wedge \omega^{m-k} \neq 0;
    \]
    \item \emph{Integrability:} there exists a section $X + \xi$ of $E$ such that
    \[
    d_H \rho := d \rho + H \wedge \rho = (X + \xi ) \cdot \rho,
    \]
    where $(X+\xi) \cdot \rho := \iota_X \rho + \xi \wedge \rho$ is the \emph{Clifford action} of sections of $E$ on spinors.
\end{itemize}
\end{definition}

The integer $\deg \Omega = k \in \{ 0, 1, \dots m \}$, defined at a point $p$, is the \emph{type} of the generalized complex structure at $p$, and it is not necessarily constant on $M$. However, it is an upper semi-continuous function, thus locally it can only decrease. Such a jump preserves the parity of $k$, that we call the \emph{parity} of $\J$. 

\begin{definition}
The type change locus of a generalized complex structure $(M,\J,H)$ is the set of points of $M$ where the type of $\J$ jumps, i.e., it is not locally constant. 
\end{definition}

Given any $2$-form $B$, the \emph{$B$-field transform} provided by $B$ is the action on pure spinors
\[
\rho \longmapsto e^B \wedge \rho.
\]
If $\rho$ is integrable with respect to $H$, then any pure spinor obtained from $\rho$ via $B$-field transform is integrable with respect to $H-dB$. Two generalized complex structures are said to be \emph{equivalent} if there exists a $B$-field transform sending one to the other.

\begin{example}
Complex structures are generalized complex. If $(M,J)$ is a complex manifold, then its canonical bundle $K:=\bigwedge^n T^*M^{1,0}$ defines a generalized complex structure locally described by $\rho = \Omega \in K$, that has constant type $m$. The converse also holds: if a generalized complex structure has constant type $m$ then it is $B$-field equivalent to one induced by a complex structure.
\end{example}

\begin{example}
Symplectic forms induce a generalized complex structure. If $(M,\omega)$ is a symplectic manifold, then the pure spinor $e^{i \omega}$ generates a canonical bundle $K$. The associated generalized complex structure $\J_\omega$ has constant type $0$. If a generalized complex structure has constant type $0$ then it is $B$-field equivalent to one induced by a symplectic form.
\end{example}

Among generalized complex structures, a special role is played by \emph{stable structures} (\cite{CG18}). Let $K$ be the canonical bundle associated to $\J$ and let $K^*$ be the corresponding anticanonical bundle. Let $s$ be a local section generating $K^*$, given by $s= \rho_0$, the zero-degree part of $\rho$. The type of $\J$ jumps precisely along the zero section, and it is constantly equal to $0$ outside of it.

\begin{definition}[\cite{CG18}, Definition 2.10]
A generalized complex structure $(M,\J,H)$ is called \emph{stable} if $s$ is transverse to the zero section $s^{-1}(0)$ of $K^*$.
\end{definition}

Outside of the type change locus, $\J$ is $B$-field equivalent to the structure induced by a symplectic form $\omega$. The type change locus is locally the non-degenerate zero set of a function, thus it has codimension $2$. 
\vspace{.2cm}

The main purpose of this paper is to carry techniques from symplectic geometry into the generalized setting. To do so, we need an extension of the notion of symplectic submanifold.

\begin{definition}[cf. \cite{GH16} and \cite{Mun18}, Definition 2.3.6]\label{j:symplectic}
Let $(M, \mathcal{J}, H)$ be a stable generalized complex structure. A submanifold $N\subset M$ is said to be \emph{$\mathcal{J}$-symplectic} if
\begin{itemize}\item there is a neighborhood $U$ of $N$ equipped with a symplectic form $\omega$ such that $N$ is a symplectic submanifold with respect to $\omega$; 
\item when restricted to $U$, the generalized complex structure $\mathcal{J}$ is $B$-field equivalent to the generalized complex structure $\mathcal{J}_{\omega}$ induced by $\omega$.
\end{itemize}
\end{definition}

\begin{remark}\label{rem:unstable}
All the constructions we describe in the rest of the paper are presented for stable generalized complex structures. However, they can be performed independently of the stability of the considered structures: we just need the notion of $\J$-symplectic submanifold, that can be easily extended to unstable structures since it depends on the behaviour of $\J$ on a neighborhood of a submanifold, and not on the global one.
\end{remark}

We give a fundamental local example. Due to the local classification of generalized complex structures (\cite{AB06, Bai12, Bai13}), every generalized complex structure is locally the product of a symplectic structure with a holomorphic Poisson structure. If we restrict to stable structures, we have the following

\begin{cor}[\cite{Bai12}, Proposition 4.2.7]\label{cor:local:stable}
Any stable generalized complex structure is locally equivalent to the structure on $\C^2 \times \R^{2m-4}$ generated by
\[
\rho_0 = (z_1 + dz^{12}) \wedge e^{i \omega_0},
\]
where $\C^2 = \{ (z_1, z_2) \}$ and $\omega_0$ is the standard symplectic form on $\R^{2m-4}$.
\end{cor}
The type change locus of $\rho_0$ consists in a single path-connected component $\{ z_1 = 0 \} = \{ 0 \} \times \C \times \R^{2m-4}$, on which $\rho_0$ has type $2$. Everywhere else, $\rho_0$ has type $0$. Indeed, taking coordinates $z_1 = r e^{i \theta}$, $z_2 = x_2 + i y_2$, such a structure is equivalent to the one induced by the symplectic form 
\begin{equation}\label{form:omega}
\Tilde{\omega} = d \log r \wedge dy^2 + d \theta \wedge dx^2 + \omega_0
\end{equation}
on the set $\{ z_1 \neq 0 \}$.
\vspace{.3cm}

\begin{example}\label{main:example}
Let $(\Sigma, \omega)$ be a symplectic $(2m-4)$-manifold. It is immediate to put a stable generalized complex structure on $D^2 \times T^2 \times \Sigma$. In fact, the product $\C^2 \times \Sigma$ admits a stable generalized complex structure generated by
\[
\rho = (z_1 + dz^{12}) \wedge e^{i \omega}.
\]
Taking the quotient of the second $\C$ factor by a lattice $\Z[i]$, $\rho$ induces a stable structure $\J$ on $D^2 \times T^2 \times \Sigma$, whose type changes from $0$ to $2$ along $\{ 0 \} \times T^2 \times \Sigma$. If $p \in D^2 \setminus \{ 0 \}$, all the submanifolds of the form $\{ p \} \times T^2 \times \Sigma$ are $\J$-symplectic, and the restriction of $\J$ to them is equivalent to the structure induced by a symplectic form $\Tilde{\omega}$ defined as in \eqref{form:omega}, replacing $\omega_0$ by $\omega$.
\end{example}

\section{Generalized Luttinger surgery on \texorpdfstring{$2m$}{}-manifolds}\label{section:luttinger:surgery}

In this section we introduce generalized Luttinger surgery on generalized complex $2m$-manifolds, with $2m \ge 6$. We begin by describing the smooth manifold that is obtained via the surgery and compute its basic topological invariants, with special attention to the fundamental group. Next, we explain how to put a stable generalized complex structure on it. We give examples and applications of the surgery, enlarging the possible diffeomorphism types of manifolds admitting stable generalized complex structures. Finally, we explain how generalized Luttinger surgery can be further extended changing the gluing diffeomorphism. Making a specific choice of diffeomorphism, we obtain generalized Gluck twist.

\subsection{The underlying smooth manifold.}\label{subsection:gluing} Let $D^2_\epsilon$ be the closed disk of radius $\epsilon$ with coordinates $r e^{i \theta^0}$, let $T^2$ be the $2$-torus with coordinates $(\theta^1, \theta^2)$ and let $(\Sigma, \omega_\Sigma)$ be a symplectic $(2m-4)$-manifold. Denote by $A(u,v)$ the annulus of radii $u,v$. Suppose that $M$ is a $2m$-manifold admitting an embedded $D^2_\epsilon \times T^2 \times \Sigma$. Consider the manifold
\begin{equation}\label{tilde:m}
\Tilde{M} := M \setminus \left( D^2_{\epsilon_0} \times T^2 \times \Sigma \right) \cup_\phi \left( D^2_1 \times T^2 \times \Sigma \right),
\end{equation}
obtained removing the embedded submanifold and gluing it back identifying $A(r_0,1) \times T^2 \times \Sigma$ with $A(\epsilon_0,\epsilon) \times T^2 \times \Sigma$ via a diffeomorphism
\[
\phi \colon A(r_0,1) \times T^2 \times \Sigma \longrightarrow A(\epsilon_0,\epsilon) \times T^2 \times \Sigma.
\]
The gluing diffeomorphism $\phi$ is defined by
\begin{equation}\label{choice:diffeo}
    \phi \left( r e^{i \theta^0}, \, \theta^1, \, \theta^2, \sigma \right) = \left( \sqrt{ 2 \log(r \epsilon )} \, e^{i \left(  p\theta^0 + a \theta^2 \right)}, \, \theta^1, \, q\theta^0 + b\theta^2, \sigma \right),
\end{equation}
where $\sigma \in \Sigma$ and $p,q,a,b$ are integers such that $pb-aq = 1$. Furthermore we can choose (cf. \cite{GH16}) $p,q,a,b$ in such a way that
\[
pbt - aq \neq 0 \qquad \text{for all $t \in [0,1]$.}
\]
Note that $\phi$ acts on $D^2 \times T^2$ as a $4$-dimensional torus surgery of multiplicity $p$ (see \cite{GS99} for the construction in the smooth setting and \cite{GH16} for the generalized complex setting). This allows to give a precise meaning to the parameters $p$, $q$, $a$, $b$. They are representing a diffeomorphism of the boundary $\partial D^2 \times T^2$ by a matrix
\[
\begin{bmatrix}
p & 0 & q \\
0 & 1 & 0 \\
a & 0 & b 
\end{bmatrix}
\in SL( 3, \Z).
\]
Using handle decomposition of $D^2 \times T^2$ it can be checked (\cite{GS99}, Ch. 4 and 8) that the resulting manifold is determined by the isotopy class of $\phi ( \partial D^2_1 \times \{ pt \})$. With this in mind, we proceed to compute the basic topological invariants of $\Tilde{M}$.

\begin{proposition}
The manifolds $M$ and $\Tilde{M}$ have the same Euler characteristic. If $2m= 4l$ for some $l\in \N$, they also have the same signature. The fundamental group of $\Tilde{M}$ is given by the quotient 
\[
\left( \pi_1 \left( M \setminus \left( D^2_{\epsilon_0} \times T^2 \times \Sigma \right) \right) \ast \pi_1(T^2) \right) / N,
\]
where $N$ is the normal subgroup generated by $\phi_* (\pi_1 (\partial D^2_1 \times T^2))$.
\end{proposition}

\begin{proof}
see \cite{BK13}, proof of Proposition 2.2.
\end{proof}

The described surgery allows to obtain manifolds with a wide choice of fundamental groups. We give two examples of that, but many more can be produced changing the parameters $p,a,b,q$. In the first example we reduce the fundamental group of the starting manifold. In the second one we show how to add torsion to it. In Sections \ref{subsection:example:surgery} and \ref{subsection:example:branched} we will endow our examples with a stable generalized complex structure (Proposition \ref{cor:symplectic:surgery}, Example \ref{generalized:add:cover}).

\begin{example}\label{ex:kill} Suppose that $X$ is a $(2m-2)$-manifold admitting an embedded $D^2 \times \Sigma$, with $\Sigma$ a $(2m-4)$-manifold. Assume that the inclusion 
\[
\pi_1 (X \setminus D^2 \times \Sigma) \hookrightarrow \pi_1 X
\]
is an isomorphism. Set $M:= X \times T^2$ and consider the manifold obtained performing the surgery \eqref{tilde:m} with parameters $p=0$, $a=1$, $b=0$, $q=1$, along
\[
D^2 \times \Sigma \times T^2 \hookrightarrow X \times T^2.
\]
Since $\phi$ is the identity on $\theta^1$, the resulting manifold is a product $\Tilde{M} \cong \Tilde{X} \times \S^1$. Moreover, $\phi$ is gluing the loop generated by $\theta^2$ in $\pi_1 (X \times T^2)$ to the loop $\partial D^2 \times \{ pt \}$. Thus, using Seifert-Van Kampen, we conclude that
\[
\pi_1 \Tilde{X} \cong \pi_1 (X \setminus D^2 \times \Sigma) \cong \pi_1 X.
\]
This reduces the fundamental group from $\pi_1 M \cong \pi_1 X \times \Z^2$ to $\pi_1 \Tilde{M} \cong \pi_1 X \times \Z$.
\end{example}

\begin{example}\label{ex:add:torsion}
Suppose that $X$ is a manifold as in Example \ref{ex:kill}. Assume that $X$ is also simply connected. Set $M:= X \times T^2$ and consider the manifold obtained performing the surgery \eqref{tilde:m} with parameters $p=0$, $a=1$, $b=0$, $q \ge 2$ along
\[
D^2 \times \Sigma \times T^2 \hookrightarrow X \times T^2.
\]
The resulting manifold $\Tilde{M}$ will be diffeomorphic to $\Tilde{X} \times \S^1$ and
\[
\pi_1 \Tilde{M} \cong \pi_1 \Tilde{X} \times \Z \cong \Z_q \times \Z.
\]
\end{example}

\subsection{The geometric structure.}\label{subsection:gls} Let $(M,\J,H)$ be a stable generalized complex structure. Let $(\Sigma, \omega_\Sigma)$ be a $(2m-4)$-symplectic manifold. Suppose that there exists an embedded, $\J$-symplectic, $D^2_\epsilon \times T^2 \times \Sigma \hookrightarrow M$ such that on it $\J$ is $B$-field equivalent to the generalized complex structure induced by the product symplectic form 
\begin{equation*}\label{prod:symp}
\Tilde{\omega} := r \, dr\, d\theta^0 + d\theta^{12} + \omega_\Sigma.
\end{equation*}

The choice of a diffeomorphism $\phi$ as in \eqref{choice:diffeo} induces a symplectic form $\phi^*(\Tilde{\omega})$ on the annulus $A(r_0,1) \times T^2 \times \Sigma$. Let $\xi$ be a radial smooth increasing function vanishing in a neighborhood of the origin and identically $1$ in a neighborhood of $[r_0,1]$. In the following lemma we use $\xi$ to extend $\phi^* (\Tilde{\omega})$ to a stable generalized complex structure defined on $D^2_\epsilon \times T^2 \times \Sigma$. Its proof is implicitly contained in the proof of Theorem 3.1 in \cite{GH16}.
\begin{lemma}\label{lemma:extension}
Let $\J_0$ be the stable generalized complex structure defined on $D^2_1 \times T^2 \times \Sigma$ by the pure spinor 
{\small
\[
\rho_0 = z_1 \, \exp \left( -\frac{p}{4} \xi ( \abs{z_1} ) \frac{dz^{1\bar 1}}{\abs{z_1}^2} -\frac{b}{2} \, dz^{2 \bar 2} + \frac{dz^1}{2 z_1} \left[ (\frac{a}{2}-q)dz^2 - (\frac{a}{2}+q) dz^{\bar 2} \right] \right)  \wedge e^{i \phi^*(\omega_\Sigma)}.
\]
}

The pullback form $\phi^*(\Tilde{\omega})$ can be extended to a stable generalized complex structure on $D^2_1 \times T^2 \times \Sigma$ that is $B$-field equivalent to $\J_0$.
\end{lemma}

The geometric construction consists in removing from $M$ the embedded $D^2_\epsilon \times T^2 \times \Sigma$ and in gluing back a copy of it as in \eqref{tilde:m} carrying a generalized complex structure that is $B$-field equivalent to $\J_0$. More precisely, let $\rho_\J$ be the pure spinor associated to $\J$, let $B_\J$ be the $2$-form defined by $\rho_\J$ and $B_0$ the one defined by $\rho_0$. Let $\rho$ be the pure spinor obtained from $\rho_0$ via the $B$-field transform
\[
\rho := e^{\Tilde{B}_0-B_0 + \phi^*(B_\J)} \rho_0,
\]
where $\Tilde{B}_0$ is a real $2$-form on $D^2_1 \times T^2 \times \Sigma$ that vanishes on a neighborhood of its boundary and agrees with $B_0$ on a neighborhood of $D^2_{r_0} \times T^2 \times \Sigma$. 
We set
\begin{equation}\label{eq:gls}
\Tilde{\rho}:=
\begin{cases}
\rho_\J &\text{on } M \setminus \left( D^2_{\epsilon_0} \times T^2 \times \Sigma \right), \\
& \\
\rho &\text{on } D^2_1 \times  T^2 \times \Sigma.
\end{cases}
\end{equation}
Since $\rho_\J$ and $\rho$ agree on the overlap, $\Tilde{\rho}$ defines a stable generalized complex structure $\Tilde{\J}$ on $\Tilde{M}$ that is integrable with respect to the $3$-form 
\[
\Tilde{H}:=
\begin{cases}
H &\text{on } M \setminus \left( D^2_{\epsilon_0} \times T^2 \times \Sigma \right), \\
 & \\
\phi^*(H) - d \Tilde{B}_0 &\text{on } D^2_1 \times T^2  \times \Sigma.
\end{cases}
\]

\begin{definition}\label{def:luttinger}
We say that the generalized complex structure $(\Tilde{M}, \Tilde{\J}, \Tilde{H})$ is obtained from $(M, \J, H)$ performing \emph{generalized Luttinger surgery} along $T^2 \times \Sigma$.
\end{definition}

\begin{remark}
We could have chosen to glue back a copy of $D^2_1 \times T^2 \times \Sigma$ carrying any stable generalized complex structure extending the pullback $\phi^*(\Tilde{\omega})$. For a fixed $\phi$, the choice of an extension not equivalent to $\rho_0$ will provide a different $\Tilde{\J}$ on the same underlying smooth manifold $\Tilde{M}$. Fixing the structure $\rho_0$ gives a definition of surgery that depends only on the choice of $\phi$. 
\end{remark}

\begin{remark}
Generalized Luttinger surgery presents similarities with another generalization of Luttinger surgery \emph{(\cite{Lut95}):} \emph{coisotropic Luttinger surgery} \emph{(\cite{BK13, Akh14, Tor16, Akh18})}. They are both performed along codimension $2$ submanifolds of the form $T^2 \times \Sigma$. However, they differ in the choice of the gluing diffeomorphism and in the geometric structures that they are gluing. 
\end{remark}

The main application of generalized Luttinger surgery will be in producing stable generalized complex structures, as made precise in the following theorem.

\begin{theorem}\label{main:surgery}
Let $(M^{2m}, \J, H)$ be a stable generalized complex structure and let $S$ be its type change locus. Let $(\Sigma, \omega_\Sigma)$ be a symplectic $(2m-4)$-manifold. Assume that there exists an embedded $T^2 \times \Sigma \hookrightarrow M$ such that
\begin{itemize} 
    \item [$(i)$] there exists a neighborhood $N_T$ of $T^2 \times \Sigma$ that is diffeomorphic to $D^2 \times T^2 \times \Sigma$;  
    \item [$(ii)$] $T^2 \times \Sigma$ is $\J$-symplectic and $\omega_{|T^2 \times \Sigma} = d \theta^{12} + \omega_\Sigma$, where $\omega$ is the symplectic form induced by $\J$ on $N_T$ up to $B$-field equivalence.
\end{itemize}
Then performing generalized Luttinger surgery along $T^2 \times \Sigma$ yields a stable generalized complex structure $(\Tilde{M}, \Tilde{\J}, \Tilde{H})$ whose type changes along $S \sqcup \left( T^2 \times \Sigma \right)$.
\end{theorem}

\begin{proof}
Let $N_T \cong D^2_\epsilon \times T^2 \times \Sigma$ be a neighborhood of $T^2 \times \Sigma$. By assumption $\J$ induces on $N_T$ a symplectic form $\omega$ that, when restricted to $T^2 \times \Sigma$, coincides with the product one. Applying Weinstein's symplectic neighborhood theorem (cf. \cite{Wei71} or \cite{MS17}, Theorem 3.4.10), up to taking a smaller $\epsilon$, $(N_T, \omega_{|N_T})$ is symplectomorphic to 
\[
(D^2_\epsilon \times T^2 \times \Sigma, \, \Tilde{\omega}:= r \, dr \, d\theta^0 + d\theta^{12} + \omega_\Sigma).
\]
We perform generalized Luttinger surgery along $T^2 \times \Sigma$ obtaining a stable generalized complex structure $(\Tilde{M}, \Tilde{\J}, \Tilde{H})$. $T^2 \times \Sigma$ is $\J$-symplectic, thus $N_T$ can be taken to be disjoint from the type change locus $S$. The type of $\Tilde{\J}$ jumps along $S$, since $\Tilde{\J}$ agrees with $\J$ on $M \setminus N_T$, and also along $T^2 \times \Sigma$ because the structure that we glued back has type jump $0 \mapsto 2$ on $\{ 0 \} \times T^2 \times \Sigma$. 
\end{proof}

\begin{remark}\label{rem:implicit}
In \emph{\cite{Mun18}}, Theorem 3.2.7, it is implicitly assumed that symplectic forms on $\J$-symplectic $6$-tori are the product form. Setting $2m=6$, $\Sigma = T^2$ in Theorem \ref{main:surgery} and precomposing with a permutation of the coordinates on $\partial D^2 \times T^4$ (cf. Remark \ref{rem:recover}), we recover the same result with the additional hypothesis made explicit.
\end{remark}

\subsection{Examples and applications.}\label{subsection:example:surgery} We apply generalized Luttinger surgery in several instances. Our main achievements consist in showing the existence of stable generalized complex structures with non-homotopy-equivalent path-connected components of their type change locus and in enlarging the class of diffeomorphism types of manifolds admitting a stable generalized complex structure.
\vspace{.3cm}

As a first application, we exploit the fact that generalized Luttinger surgery can be performed along $T^2 \times \Sigma$ to produce stable generalized complex structures whose path-connected components of the type change locus are not pairwise homotopy-equivalent.

\begin{theorem}\label{cor:non:homotopy}
There exist stable generalized complex structures with non-homotopy-equivalent path-connected components of their type change locus.
\end{theorem}

\begin{proof}
Denote by $X$ the manifold described in \cite{BK09}, Theorem 20. Recall that $X$ is a symplectic $4$-manifold containing an embedded $2$-torus $T$ and an embedded surface $\Sigma$ of genus $2$ that can be taken to be symplectic and disjoint. Both have a neighborhood diffeomorphic to $D^2 \times T$, respectively $D^2 \times \Sigma$. Consider on $M := X \times T^2$ the generalized complex structure induced by the product symplectic form and perform two generalized Luttinger surgery on the embedded, $\J$-symplectic, disjoint $T\times T^2$ and $\Sigma \times T^2$. This provides a stable generalized complex structure $(\Tilde{M}, \Tilde{\J}, \Tilde{H})$ on a $6$-manifold whose type change locus has two path connected components that are not homotopy-equivalent. Stable structures in any even dimension $2m \ge 8$ can be obtained taking products of $\Tilde{M}$ with a generalized complex $(2m-6)$-manifold of constant type $0$.
\end{proof}

\begin{remark}\label{rem:diffeo:type}
We point out that in dimension $4$ each path-connected component of the type change locus of a stable generalized complex structure is diffeomorphic to a $2$-torus, inheriting the structure of an elliptic curve \emph{(\cite{CG07}}, Theorem 2.1), thus the phenomenon described in Theorem \ref{cor:non:homotopy} is exclusive of stable structures of dimension $2m \ge 6$.
\end{remark}

Ideally, one would be able to determine the diffeomorphism type of the manifold $\Tilde{M}$ obtained via generalized Luttinger surgery, with the purpose of producing generalized complex structures on manifolds that were not previously known to admit one. In the rest of this section we obtain some partial result in this direction.

Perform a single generalized Luttinger surgery on a generalized complex manifold of the form $M \times T^2$. Since the surgery acts as the identity on the first coordinate of the $2$-torus, the resulting manifold can be written as a product $\Tilde{M} \times \S^1$. More precisely, we have the following

\begin{proposition}\label{cor:symplectic:surgery}
Let $(M^{2m-2},\J,H)$ be a stable generalized complex structure. Let $(\Sigma, \omega_\Sigma)$ be a symplectic $(2m-4)$-manifold. Assume that there exists an embedded, $\J$-symplectic $\Sigma \hookrightarrow M$ with a neighborhood diffeomorphic to $D^2 \times \Sigma$. Fix on $T^2$ a generalized complex structure induced by a symplectic form. Then performing generalized Luttinger surgery along the submanifold
\[
\Sigma \times T^2 \hookrightarrow M \times T^2,
\]
endowed with the product structure, yields a stable generalized complex structure
\[
(\Tilde{M}^{2m-1} \times \S^1, \Tilde{\J}, \Tilde{H}).
\]
\end{proposition}

In dimension $2m = 6$, under some assumptions on the manifold $M$ and the diffeomorphism $\phi$, it is possible to obtain more information on the diffeomorphism type of $\Tilde{M}$.

\begin{proposition}\label{partial:diffeo:type}
Let $(M^4,\J,H)$ be a stable generalized complex structure as in Proposition \ref{cor:symplectic:surgery}. Assume that both $M$  and the complement $M \setminus \left( D^2 \times \Sigma \right)$ are simply connected. Let $g$ be the genus of $\Sigma$ and set $k = b_2(M) +2g-1$. Perform a generalized Luttinger surgery with parameters $p=0$, $a = 1$, $q = 1$, $b=0$ along the submanifold
\[
\Sigma \times T^2 \hookrightarrow M \times T^2
\]
to obtain a stable generalized complex structure on $\Tilde{M} \times \S^1$. The resulting manifold is either diffeomorphic to
\begin{equation}\label{conn:sum}
\left( \#_k \, \S^2 \times \S^3 \right ) \times \S^1,
\end{equation}
when $\Tilde{M}$ is spin, or to
\begin{equation}\label{twisted:conn:sum}
\left( \S^2 \Tilde{\times} \S^3 \# \left( \#_{k-1} \, \S^2 \times \S^3 \right ) \right) \times \S^1,
\end{equation}
when $\Tilde{M}$ is non-spin.

\end{proposition}
\begin{proof}
Using Seifert-van Kampen's theorem (cf. Example \ref{ex:kill}), we see that $\Tilde{M}$ is simply connected. Using Poincaré duality and a Mayer-Vietoris argument, it is straightforward to compute
\[
H_2 (\Tilde{M}; \Z) \cong H_3 (\Tilde{M}; \Z) \cong \Z^k,
\]
where $k = b_2(X) + 2g -1$. We conclude the proof applying classification results for simply connected $5$-manifolds proved in \cite{Sma62} and \cite{Bar65}, that allow to determine the diffeomorphism type of $\Tilde{M}$ knowing its homology, its linking form and whether or not it is spin. If $\Tilde{M}$ is spin, then
\[
\Tilde{M} \cong \#_k \, \S^2 \times \S^3.
\]
If $\Tilde{M}$ is non-spin, then
\[
\Tilde{M} \cong \S^2 \Tilde{\times} \S^3 \# \left( \#_{k-1} \, \S^2 \times \S^3 \right ).
\]
\end{proof}

Using Proposition \ref{partial:diffeo:type}, producing stable generalized complex structures on $6$-manifolds with diffeomorphism type \eqref{conn:sum} or \eqref{twisted:conn:sum} reduces to a symplectic geography problem, as shown in the next example.

\begin{example}
Let $X$ be a simply connected, symplectic $4$-manifold containing an embedded symplectic surface $\Sigma$ of genus $g \ge 0$ with trivial normal bundle and such that $X \setminus D^2 \times \Sigma$ is simply connected. Suppose that $\Sigma$ intersects transversely a surface $S$ of self-intersection $-1$, at a single point $x$. In particular, the normal bundle of $S$ is not trivial. \\
Endow $X \times T^2$ with a product symplectic form and perform generalized Luttinger surgery along $\Sigma \times T^2$ as in Proposition \ref{partial:diffeo:type} obtaining a stable generalized complex structure on $\Tilde{M} \times \S^1$. The manifold $\Tilde{M}$ is simply connected and non-spin, since the surface obtained removing from $S$ a disk around $x$ and gluing it back via generalized Luttinger surgery still has non-trivial normal bundle. This produces a stable structure on 
\[
\left( \S^2 \Tilde{\times} \S^3 \# \left( \#_{b_2(X) + 2g -2} \, \S^2 \times \S^3 \right ) \right) \times \S^1.
\]
Taking explicit examples of manifolds satisfying our assumptions we have that:
\begin{itemize}
    \item the elliptic surface $E(1)$, whit $\Sigma$ given by the generic fiber and $S$ by a cusp neighborhood, yields a stable generalized complex structure on
    \[
    \left( \S^2 \Tilde{\times} \S^3 \# \left( \#_{10} \, \S^2 \times \S^3 \right ) \right) \times \S^1;
    \]
    \item the manifold $X$ obtained in \cite{BK09}, Theorem 13, with $\Sigma$ given by a surface of genus $2$ and $S$ by a torus denoted by $H_1$, yields a stable structure on
    \[
    \left( \S^2 \Tilde{\times} \S^3 \# \left( \#_6 \, \S^2 \times \S^3 \right ) \right) \times \S^1.
    \]
\end{itemize}
\end{example}
The construction can be slightly modified asking $\Sigma$ to be disjoint from $S$, leading to the same conclusion as above. For example, this is the case of the manifold $B$ obtained in \cite{BK09}, Theorem 18, with $\Sigma$ and $S$ both given by disjoint tori. This produces a stable structure on 
\[
\left( \S^2 \Tilde{\times} \S^3 \# \left( \#_8 \, \S^2 \times \S^3 \right ) \right) \times \S^1.
\]

\section{Changing the gluing diffeomorphism.}\label{subsection:change:gluing} We show how generalized Luttinger surgery can be modified choosing a gluing diffeomorphism different from \eqref{choice:diffeo}, that does not necessarily act as the identity on $\Sigma$. The general procedure is as follows:
\begin{itemize}
    \item Pick a diffeomorphism $\phi \colon A(r_0, 1) \times T^2 \times \Sigma \rightarrow A(\epsilon_0, \epsilon) \times T^2 \times \Sigma$. This gives a lot of freedom of choice. For example, $\phi$ can act on $D^2 \times T^2$ as torus surgery of multiplicity $p$ and on $\Sigma$ as any self-diffeomorphism.
    
    \item Fix on $D^2_\epsilon \times T^2 \times \Sigma$ a symplectic form $\omega$. Assume that we can extend $\phi^*( \omega)$ to a stable generalized complex structure $\J_0$ on $D^2_1 \times T^2 \times \Sigma$.
    
    \item Suppose that $(M, \J, H)$ is a stable generalized complex manifold admitting an embedded, $\J$-symplectic $D^2 \times T^2 \times \Sigma$ on which $\J$ is $B$-field equivalent to the generalized structure induced by $\omega$. Remove $D^2 \times T^2 \times \Sigma$ and glue it back via $\phi$, carrying the generalized complex structure $\J_0$ extending $\phi^* (\omega)$, as we did in Section \ref{subsection:gls}. This endows the resulting manifold $\Tilde{M}$ with a stable generalized complex structure $\Tilde{J}$.
\end{itemize}

\subsection{Generalized Gluck twist}\label{subsection:gluck}
As an example of the general procedure, we pick a specific diffeomorphism and obtain a surgery that we call \emph{generalized Gluck twist}. Suppose that 
\[
(\Sigma, \, \omega_\Sigma) \cong (R \times \S^2, \, \omega_R + \omega_{\S^2}),
\]
where $(R, \omega_R)$ is a symplectic $(2m-6)$ manifold. The gluing diffeomorphism is obtained combining generalized Luttinger surgery and Gluck twist as we explain now.\\
Denote by $(h, \theta^S)$ cylindrical coordinates on the sphere. Suppose that $X$ is a $4$-manifold admitting an embedded $2$-sphere with trivial normal bundle. The manifold obtained from $X$ via Gluck twist is the manifold
\[
\Tilde{X}:= \left( X \setminus D^2 \times \S^2 \right) \cup_\tau D^2 \times \S^2,
\]
where $\tau$ is the diffeomorphism of the boundary $\partial D^2 \times \S^2$ given by
\[
\tau ( \theta^0, \, h, \, \theta^S) = ( \theta^0, \, h, \, \theta^S + \theta^0).
\]
Suppose now that $(M^{2m}, \J, H)$ is a stable generalized complex manifold admitting an embedded, $\J$-symplectic $D^2_\epsilon \times T^2 \times R \times \S^2$ such that on it $\J$ is $B$-field equivalent to the structure induced by the symplectic form
\[
\omega:= r \, dr \, d\theta^0 + d \theta^{12} + \omega_R + dh \, d\theta^S .
\]
Consider the manifold
\begin{equation*}\label{gluing:gluck}
\Tilde{M} := M \setminus \left( D^2_{\epsilon_0} \times T^2 \times R \times \S^2 \right) \cup_\phi (D^2_1 \times T^2 \times R \times \S^2),
\end{equation*}
obtained removing the embedded submanifold and gluing it back identifying the annuli via the diffeomorphism $\phi$, defined by
{\small
\[
  \phi \left( r e^{i \theta^0}, \, \theta^1, \, \theta^2, \sigma, h, \theta^S \right) = \left( \sqrt{ 2 \log(r \epsilon )} \, e^{i \left(  p\theta^0 + a \theta^2 \right)}, \, \theta^1, \, q\theta^0 + b\theta^2, \varphi (\sigma), h, \theta^S + \theta^0 \right),
\]
}

where $\sigma \in R$, $p,q,a,b$, are as in Section \ref{subsection:gls} and $\varphi$ is any self-diffeomorphism of $R$. Note that $\phi$ restricts to torus surgery of multiplicity $p$ on $D^2 \times T^2$ and to Gluck twist on $D^2 \times \S^2$. The symplectic form $\phi^*(\omega)$ can be extended from $A(r_0, 1) \times T^2 \times R \times \S^2 $ to a stable generalized complex structure on $D^2_1 \times T^2 \times R \times \S^2$, obtained modifying the pure spinor $\rho_1$ defined in \cite{Mun18}, proof of Theorem 3.3.2. We set
\begin{align*}
    \Tilde{\rho}_1 = z_1 \, \exp \Big( &-\frac{p}{4} \xi ( \abs{z_1}^2 ) \frac{dz^{1\bar 1}}{\abs{z_1}^2} -\frac{b}{2} \, dz^{2 \bar 2} +\\
    & \frac{dz^1}{2 z_1} \left[ (\frac{a}{2}-q)dz^2 - (\frac{a}{2}+q) dz^{\bar 2} \right] -\\
    & \frac{2}{(1 + \abs{z_3}^2)^2} \, dz^2 \left[ \bar{z}_3 dz^3 + z_3 dz^{\bar 3} \right] \Big),
\end{align*}
where $z_1$, $z_2$ are complex coordinates on $D^2$, $T^2$ respectively, and $z_3$ is the complex coordinate on $\S^2$ given by the stereographic projection.

\begin{lemma}
Consider on $D^2_1 \times T^2 \times R \times \S^2$ the stable generalized complex structure $\J_0$ defined by the pure spinor
\[
\rho_0 = \Tilde{\rho}_1 \wedge e^{i \phi^*(\omega_R)}.
\]
The pullback form $\phi^*(\omega)$ can be extended to a stable generalized complex structure on $D^2_1 \times T^2 \times R \times \S^2$ that is $B$-field equivalent to $\J_0$.
\end{lemma}

Following Section \ref{subsection:gls}, we can glue $\J$ and $\J_0$ to endow the manifold $\Tilde{M}$ with a stable generalized complex structure $\Tilde{\J}$.

\begin{definition}
We say that the generalized complex structure $(\Tilde{M}, \Tilde{\J}, \Tilde{H})$ is obtained from $(M, \J, H)$ performing \emph{generalized Gluck twist} along $T^2 \times R \times \S^2$.
\end{definition}

Several results proved for generalized Luttinger surgery hold also for generalized Gluck twist with analogous proofs. For instance, Theorem \ref{main:surgery} corresponds to the following result.

\begin{theorem}
Let $(M^{2m}, \J, H)$ be a stable generalized complex structure and let $S$ be its type change locus. Let $(R, \omega_R)$ be a symplectic $(2m-6)$-manifold. Assume that there exists an embedded $T^2 \times R \times \S^2 \hookrightarrow M$ such that
\begin{itemize} 
    \item [$(i)$] there exists a neighborhood $N_T$ of $T^2 \times R \times \S^2$ that is diffeomorphic to $D^2 \times T^2 \times R \times \S^2$;  
    \item [$(ii)$] $T^2 \times R \times \S^2$ is $\J$-symplectic and $\omega_{|T^2 \times R \times \S^2} = d \theta^{12} + dh \, d\theta^S + \omega_R$, where $\omega$ is the symplectic form induced by $\J$ on $N_T$ up to $B$-field equivalence.
\end{itemize}
Then performing generalized Gluck twist along $T^2 \times R \times \S^2$ yields a stable generalized complex structure $(\Tilde{M}, \Tilde{\J}, \Tilde{H})$ whose type changes along $S \sqcup ( T^2 \times R \times \S^2)$.
\end{theorem}

\begin{example}
Let $(X^{2m}, \J, H)$ be a stable generalized complex structure. Assume that $X$ contains an embedded, $\J$-symplectic $T^2 \times R$ with trivial normal bundle. Endow $M:= X \times \S^2$ with the product generalized complex structure induced by a symplectic form on $\S^2$. Performing generalized Gluck twist along $T^2 \times R \times \S^2 \hookrightarrow M$ yields a stable generalized complex structure on a manifold $\Tilde{M}$ which has the structure of a fiber bundle with fiber $\S^2$ and base $\Tilde{X}$, where $\Tilde{X}$ is the manifold obtained performing the surgery along $T^2 \times R \hookrightarrow X$ induced by the restriction of $\phi$ to the boundary $\partial D^2 \times T^2 \times R$.
\end{example}

\begin{remark}\label{rem:recover}
In some special cases, our surgeries recover known techniques. Allowing in Section \ref{subsection:gluing} $\Sigma$ to be a point, generalized Luttinger surgery coincides with torus surgery of multiplicity $p$. Similarly, allowing in Section \ref{subsection:gluck} $R$ to be a point, generalized Gluck twist reduces to the surgery of Theorem 3.3.2 in \cite{Mun18}.
\end{remark}

\section{Branched coverings of generalized complex manifolds}\label{section:branched:coverings}

We begin this section recalling the notion of branched coverings and we give some basic example. Then, we add the geometric structure and show that branched coverings of stable generalized complex manifolds with $\J$-symplectic branch locus are still stable generalized complex. Finally we explain how to obtain stable generalized complex structures combining the techniques we described.

\subsection{Branched coverings.} Let $M$ and $\Tilde{M}$ be $n$-manifolds and $d \in \N$.

\begin{definition} 
A \emph{$d$-fold branched covering} of $M$ is a proper, smooth map $\pi \colon \Tilde{M} \longrightarrow M$ with critical set $B\subset M$ such that
\begin{itemize}
    \item [$(i)$] outside of $B$, the restriction of $\pi$
    \[
    \pi|_{\Tilde{M} \setminus \pi^{-1}(B)} \colon \Tilde{M} \setminus \pi^{-1}(B) \longrightarrow M \setminus B
    \]
     is a covering map of degree $d$;
    \item[$(ii)$] for each $p\in \pi^{-1}(B)$ there are local coordinate charts $U$, $V \longrightarrow \C\times \R^{n - 2}_+$ around $p$ and $\pi(p)$, respectively, on which $\pi$ has the form
    \[
    (z, x) \longmapsto (z^t, x)
    \]
    for some $t\in \N$, $t \ge 2$. 
\end{itemize}
The set $B$ is called the \emph{branch locus} of $\pi$ and the integer $t$ is the \emph{branching index} of $\pi$ at the point $p \in B$. 
\end{definition}

We give some basic example of branched coverings of smooth manifolds that will be useful in building explicit examples in generalized complex geometry. 

\begin{example}[Branched covering of product of surfaces]\label{example:branched:surfaces}
Let $\Sigma$ and $\Tilde{\Sigma}$ be closed orientable surfaces of genus $g$, $\Tilde{g} \ge 0$, respectively. If there exists a branched covering  
\begin{equation}
\pi \colon \Tilde{\Sigma} \longrightarrow \Sigma,
\end{equation}
the Riemann-Hurwitz formula implies that $\Tilde{g} \ge g$. Viceversa, if $\Tilde{g} \ge g$, $\Tilde{\Sigma}$ can always be realized as a branched covering of $\Sigma$. In such a case, its branch locus consists in a finite number $k$ of isolated points $\{ x_1, \dots, x_k\}$. As a consequence, we can build branched coverings of $4$-manifolds that are products of two surfaces
\[
\psi \colon \Sigma_h \times \Tilde{\Sigma} \longrightarrow \Sigma_h \times \Sigma,
\]
where $\psi$ is obtained extending $\pi$ to act as the identity on $\Sigma_h$. Their branch locus is $B = \Sigma_h \times \{ x_1, \dots, x_k\}$.
\end{example}

\begin{example}[\cite{GS99}, p. 245]\label{ex:elliptic}
Let $E(n)$ be a simply connected elliptic surface. There is a 2-fold branched covering
\[
E(2n)\rightarrow E(n)
\]
with branch locus a pair of generic regular fibers.
\end{example}

\subsection{The geometric structure.} We start from a given stable generalized complex structure on a $2m$-manifold $M$, with $2m \ge 4$. We use a branched covering $\pi \colon \Tilde{M} \rightarrow M$ to build a stable generalized complex structure on $\Tilde{M}$. This is straightforward when the branch locus is empty, i.e. for genuine covering maps, as proved in the following

\begin{lemma}\label{GCS:pullback}
Let $(M, \J, H)$ be a stable generalized complex structure whose type change locus has $k \in \N$ path-connected components. Let $\pi \colon \Tilde{M} \longrightarrow M$ be a $d$-fold covering map of $M$. Then there exists a stable generalized complex structure $(\Tilde{M}, \Tilde{\J}, \Tilde{H})$ whose type change locus has $d \cdot k$ path-connected components.
\end{lemma}

\begin{proof}
Let $\rho$ be the pure spinor defined by $\J$. Then $\Tilde{\rho} := \pi^* \rho$ is well-defined since $\pi$ is a local diffeomorphism and defines on $\Tilde{M}$ a generalized complex structure $\Tilde{\J}$ integrable with respect to $\Tilde{H}:= \pi^* H$. If $\rho$ has type $t$ at $p \in M$, then $\Tilde{\rho}$ has type $t$ at $\pi^{-1}(p)$. Thus the type of $\Tilde{\J}$ jumps precisely along the preimage of the type change locus of $\J$, that has $d \cdot k$ path-connected components. For the same reason, since $\J$ is stable so is $\Tilde{\J}$.
\end{proof}

A similar statement can be proved for branched coverings. Having underlined the role of $\J$-symplectic submanifolds, the proof reduces to using Lemma \ref{GCS:pullback} and to rephrasing a result in \cite{Gom98} in the generalized complex language.

\begin{theorem}[\cite{Gom98}, Lemma 1] \label{main:branched:coverings}
Let $(M^{2m}, \J,H)$ be a stable generalized complex structure whose type change locus has $k \in \N$ path-connected components. Let $B$ be a $\J$-symplectic submanifold of $M$, and let $\pi: \Tilde{M} \rightarrow M$ be a $d$-fold branched covering with branch locus $B$. Then there exists a stable generalized complex structure $(\Tilde{M}, \Tilde{\J}, \Tilde{H})$ whose type change locus has $d \cdot k$ path-connected components.
\end{theorem}

\begin{proof}
Lemma 1 in \cite{Gom98} establishes that if $(M, \omega)$ is a closed, symplectic manifold and $\pi \colon \Tilde{M} \rightarrow M$ is a $d$-fold branched covering with symplectic branch locus $B$ of codimension $2$, then there exist a neighborhood $V$ of $B$ and a symplectic form $\Tilde{\omega}$ on $\Tilde{M}$ that agrees with $\pi^* \omega$ on $\Tilde{M} \setminus \pi^{-1} (V)$, with $[\Tilde{\omega}] = \pi^* [\omega]$. In other words, $\pi^* \omega$ can be extended from $\Tilde{M} \setminus \pi^{-1} (V)$ to the preimage of $V$.\\
In our case, let $N$ be a neighborhood of $B$ and let $\omega$ be a symplectic form on $N$ inducing a generalized complex structure equivalent to $\J$, up to the $B$-transform $e^{B_0}$. Since the construction in \cite{Gom98} is performed on a neighborhood of $B$, it can be applied to $(N,\omega)$ to find a neighborhood $V \subset N$ of $B$ and a symplectic form $\Tilde{\omega}$ extending $\pi^* \omega$ from $\pi^{-1}(N\setminus V)$ to $\pi^{-1}(N)$. Let $\rho_\J$ be the pure spinor associated to $\J$, and consider the pure spinor defined by
\[
\Tilde{\rho} := 
\begin{cases}
\pi^* \rho_\J & \text{on } \Tilde{M}\setminus \pi^{-1} (V), \\
 & \\
e^{\pi^* \Tilde{B}_0 + \pi^*(B_\J) + i \Tilde{\omega}} & \text{on } \pi^{-1} (N),
\end{cases}
\]
where $B_\J$ is the $2$-form defined by $\rho_\J$ and $\Tilde{B}_0$ is an extension of $B_0$ as in the proof of Theorem \ref{main:surgery}.
Due to Lemma \ref{GCS:pullback} and Lemma 1 in \cite{Gom98}, $\Tilde{\rho}$ is a well-defined stable generalized complex structure on $\Tilde{M}$, that is integrable with respect to
\[
\Tilde{H}:=
\begin{cases}
\pi^* H &\text{on } \Tilde{M}\setminus \pi^{-1} (V), \\
 & \\
\pi^* H  - d (\pi^* \Tilde{B}_0)  &\text{on } \pi^{-1} (N).
\end{cases}
\]
Since $\Tilde{\rho}$ has type $0$ on $\pi^{-1}(B)$, the second statement from Lemma \ref{GCS:pullback} completes the proof of the theorem.
\end{proof}

\subsection{Examples and applications.} \label{subsection:example:branched} We use coverings and branched coverings, combined with other techniques, to exemplify how to obtain stable generalized complex structures.
\vspace{.3cm}

Our first example uses generalized Luttinger surgery together with coverings of generalized complex manifolds to produce stable generalized complex structures on a variety of manifolds.

\begin{example}\label{generalized:add:cover} Let $2m \ge 6$. Consider a stable generalized complex $2m$-manifold $(X, \J, H)$ admitting an embedded, $\J$-symplectic $(2m-2)$-manifold $\Sigma$. Fix a generalized complex structure on $T^2$ induced by a symplectic form and endow $M = X \times T^2$ with the product structure. Perform generalized Luttinger surgery along $\Sigma \times T^2$ to add torsion to the fundamental group (e.g. Example \ref{ex:add:torsion}). The resulting manifold will admit a stable generalized complex structure and any of its coverings will also admit one, according to Lemma \ref{GCS:pullback}. This construction expands the possible diffeomorphism types of stable generalized complex manifolds.
\end{example}

A similar construction works also in the $4$-dimensional case, involving symplectic sum (\cite{Gom95}) and replacing generalized Luttinger surgery with torus surgery. We give an example of how this can be done.

\begin{example}\label{ex:add}
The procedure described in \cite{Tor14} (Lemma 13 and 14) allows to build, using symplectic sum and torus surgery, symplectic manifolds with fundamental group $G \in \left \{ \{1 \}, \Z_p, \Z_p \oplus \Z_q, \Z \oplus \Z_q, \Z, \Z^2  \right \}$ (see also \cite{BK08, BK09, TY15}). The construction can be repeated verbatim for stable generalized complex manifolds, as long as the submanifolds along which we perform the symplectic sum and the torus surgeries are $\J$-symplectic. Furthermore, taking coverings (cf. \cite{Tor14}, Remark 20), we obtain more stable generalized complex structures thanks to Lemma \ref{GCS:pullback}.
\end{example}

Thanks to Theorem \ref{main:branched:coverings}, finding generalized complex structures amounts to finding branched coverings with $\J$-symplectic branch locus. We show here two examples. However, we point out that in general checking that a submanifold is $J$-symplectic is not an easy task.

\begin{example}[Generalized complex structures on $T^2 \times \Sigma_g$] \label{gcs:surfaces}
Fix $g \ge 0$. In \cite{TY14} it was proved that $T^2 \times \Sigma_g$ admits a family of stable generalized complex structures $\{ \J_n \}$ whose type change locus has $n$ path-connected components. In particular, consider any of such structures $\J$ on $T^2 \times \S^2$ and let $S$ be its type change locus. As long as $ x \notin S$, $T^2 \times \S^2$ contains $T^2 \times \{ x \}$ as a $\J$-symplectic submanifold. Thus we can use the branched covering from Example \ref{example:branched:surfaces} to put a stable generalized complex structure on $T^2 \times \Sigma_g$ for $g \ge 1$. These structures can also be obtained using only torus surgeries as described in \cite{TY14}.
\end{example}

\begin{example}[Generalized complex structures on elliptic surfaces]
Stable generalized complex structures on elliptic surfaces have been built in \cite{CG09, Tor12, TY14, GH16, CKW20, CKW22}. Branched coverings provide a new way of finding stable structures on $E(2n)$. Indeed, the elliptic surface $E(n)$ admits a stable generalized complex structure for which the generic fibers that are disjoint from the type change locus can be taken to be $\J$-symplectic (\cite{CG09}, Example 5.3). Taking the branched covering of Example \ref{ex:elliptic}, we obtain a stable generalized complex structure on $E(2n)$.
\end{example}

\printbibliography

\end{document}